\documentclass[11pt]{article} \topmargin -0.25cm \oddsidemargin 0.5cm
\usepackage{amsmath}
\usepackage{amssymb}
\newcommand{\s}{\sum}
\newcommand{\R}{{\mathbb R}}
\newcommand{\Rp}{{\mathbb R^+}}
\newcommand{\Md}{{\mathbb M^d}}
\newcommand{\D}{{\mathbb D}(\Rp,\R)}

\newcommand{\DD}{{\mathbb D}(\Rp,{\mathbb R}^{2})}

\newcommand{\Dppd}{{\mathbb D}(\Rp,{\mathbb R}^{d^2+2d+1})}
\newcommand{\Dpppd}{{\mathbb D}(\Rp,{\mathbb R}^{3d+1})}

\newcommand{\Dpd}{{\mathbb D}(\Rp,{\mathbb R}^{d+1})}
\newcommand{\Dpddd}{{\mathbb D}(\Rp,{\mathbb R}^{3d+1})}
\newcommand{\Dcd}{{\mathbb D}(\Rp,{\mathbb R}^{4d})}
\newcommand{\Dd}{{\mathbb D}(\Rp,{\mathbb R}^{d})}
\newcommand{\Dbd}{{\mathbb D}(\Rp,{\mathbb R}^{\bar d})}
\newcommand{\Dbdd}{{\mathbb D}(\Rp,{\mathbb R}^{2\bar d})}
\newcommand{\Dbdwad}{{\mathbb D}(\Rp,{\mathbb R}^{{\bar d}^2})}
\newcommand{\Ddd}{{\mathbb D}(\Rp,{\mathbb R}^{2d})}

\newcommand{\Ddwadd}{{\mathbb D}(\Rp,{\mathbb R}^{d^2+2d})}

\newcommand{\Ddwad}{{\mathbb D}(\Rp,{\mathbb R}^{d^2+d})}
\newcommand{\Dddt}{{\mathbb D}(\Rp,{\mathbb R}^{2d+1})}
\newcommand{\Ddddt}{{\mathbb D}(\Rp,{\mathbb R}^{3d+1})}
\newcommand{\DMd}{{\mathbb D}(\Rp,{\mathbb M}^{d})}

\newcommand{\Ddddd}{{\mathbb D}(\Rp,{\mathbb R}^{4d})}

\newcommand{\Dddddj}{{\mathbb D}(\Rp,{\mathbb R}^{4d+1})}

\newcommand{\lra}{\longrightarrow}
\newcommand{\BH}{B^H}
\newcommand{\ZH}{Z^{H}}
\newcommand{\BB}{Z^{H}}
\newcommand{\LH}[1]{\mathbb{L}^{1/H}_{#1}}

\newcommand{\N}{{\mathbb N}}
\newcommand{\No}{{\mathbb N}\cup\{0\}}

\newcommand{\RR}{{\mathbb R}}
\newcommand{\wlr}{\geq}
\newcommand{\Rd}{{{{\mathbb R}^d}}}

\newcommand{\RRp}{{{{\mathbb R}^+}}}

\newcommand{\mr}{\leq}
\newtheorem{theorem}{\bf Theorem}[section]
\newtheorem{proposition}[theorem]{\bf Proposition}
\newtheorem{lemma}[theorem]{\bf Lemma}
\newtheorem{definition}[theorem]{\bf Definition}
\newtheorem{corollary}[theorem]{\bf Corollary}
\newtheorem{remark}[theorem]{\bf Remark}
\newenvironment{proof}{{\sc Proof.}}{\hfill $\Box$}

\begin{document}\title{\bf SDEs with  constraints driven by processes with
bounded $p$-variation}
\author{Adrian Falkowski and Leszek S\l omi\'nski
\footnote{Corresponding author. E-mail address: leszeks@mat.umk.pl;
Tel.: +48-566112954; fax: +48-566112987.}\\
 \small Faculty of Mathematics and Computer Science,
Nicolaus Copernicus University,\\
\small ul. Chopina 12/18, 87-100 Toru\'n, Poland}
\date{}
\maketitle
\begin{abstract}
We study the existence, uniqueness and approximation of solutions of
stochastic differential equations with constraints driven by
processes  with bounded $p$-variation. Our main tool are new
estimates showing  Lipschitz continuity of the deterministic Skorokhod problem
in $p$-variation norm. Applications to fractional SDEs  with constraints are given.
\end{abstract}

\section{Introduction}

In the present paper we study the problems of existence,
uniqueness and approximation of solutions of finite-dimensional
stochastic differential equations  \- (SDEs) with constraints driven
by general processes with bounded $p$-variation, $p\geq1$. More
precisely, let $f:\Rd\rightarrow\Rd$,
$g:\Rd\rightarrow\Rd\otimes\Rd$ be measurable functions, $A$ be a
one-dimensional process with locally  bounded variation and $Z$ be
a $d$-dimensional  process with locally bounded  $p$-variation. We
consider SDEs with reflecting boundary condition of the form
\begin{equation}
\label{eq1.1} X_t = X_0 + \int_0^t f(X_{s-})\,dA_s+\int_0^t
g(X_{s-})\,dZ_s+K_t, \quad t\in\mathbb R^+.
\end{equation}
By a solution to (\ref{eq1.1}) we mean a pair $(X,K)$ consisting
of a process $X$ living over a given $d$-dimensional barrier
process $L$ and a $d$-dimensional process $K$, called regulator
term, whose each component $K^i$ is nondecreasing and increases
only when  $X^i$  is living on  $L^i$ (for details see Section 3). Equation (\ref{eq1.1}) is
called the Skorokhod SDE in analogy with the case $L=0$ first
discussed by  Skorokhod \cite{sk} for a standard Brownian motion
in place of $Z$ and $A_t=t$, $t\in\Rp$. Next, many attempts have
been made to extend Skorokhod's results to larger class of domains
or larger class of driving processes (see, e.g.,
\cite{cl,di,ls,s4,ta}). This kind of equations have many   applications, for instance in queueing
 systems, seismic reliability analysis and  finance (see, e.g.,\cite{as,dr,KS,ss}).
In recent papers by Besalu and Rovira
\cite{br} and  Ferrante and Rovira \cite{fr1} SDE with
non-negativity constraints driven by fractional Brownian motion
$B^H$  with Hurst index $H>1/2$ and $A_t=t$, $t\in\Rp$, is
studied. This equation is a particular case of (\ref{eq1.1})
because $B^H$ has locally bounded $p$-variation for $p>1/H$. In
the main  theorem of \cite{fr1} the existence of a solution is
proved  under the assumption that the coefficients $f,g$ are
Lipschitz continuous. The proof is based on a quite natural in the
context of SDEs driven by $B^H$ technics based on
$\lambda$-H\"older norms. Unfortunately, in \cite{fr1} it is only
shown that the  solution is unique for some small time interval.
To our knowledge, global uniqueness for fractional SDEs with
constraints is still an open problem.
In contrast to \cite{fr1}, in our paper we use  $p$-variation norm (for the theory of functions of  $p$-variation and its various applications  see, e.g., \cite{dn,dn1}).

In our paper we consider two conditions: continuity and linear
growth of $f$ and H\"older continuity of $g$ (condition (H1))  and
local Lipschitz continuity of $f$  and local H\"older continuity
of the derivative of each component $g_{i,j}$ $i,j=1,\dots,d$
(condition (H2)) (see Section 3). We show that under (H1) and (H2)
there exists a unique (globally in time) solution to (\ref{eq1.1}), which can be approximated by some natural
approximation  schemes.

The paper is organised as follows.

In Section 2 we consider the  deterministic Skorokhod problem
$x=y+k$ associated with $y\in\Dd$ and time dependent lower barrier
$l\in\Dd$ with $l_0\leq y_0$. We show that the mapping
$(y,l)\mapsto(x,k)$ is Lipschitz continuous  in $p$-variation
norm. In fact, we show that if $(x,k)$, is a solution associated
with ${y}\in\Dd$ and barrier $l\in\Dd$ and $(x',k')$,  is a
solution associated with ${y'}\in\Dd$ and barrier $l'\in\Dd$ then
for any $T\in\Rp$,
\begin{equation}\label{eq1.2}
\bar V_p(x-x')_T\leq(d+1)\bar V_p(y-y')_T +d\bar V_p(l-l')_T
\end{equation} and
\begin{equation}
\label{eq1.3} \bar V_p(k-k')_T\leq d\bar V_p(y-y')_T +d\bar V_p(l-l')_T.
\end{equation}
It is worth noting here that in \cite[Remark 3.6]{fr1} it is
observed that $(y,l)\mapsto(x,k)$ is not Lipschitz continuous in
the $\lambda$-H\"older norm and for that reason in \cite{fr1} the
authors were not able to obtain global uniqueness.

In  Section 3 we consider a deterministic counterpart to
(\ref{eq1.1}). We prove that under (H1) the deterministic equation
has a solution. If moreover (H2) is satisfied that it is unique.
Then we show convergence of some natural approximation schemes for
a deterministic equation of the form (\ref{eq1.1}). In the proofs
of convergence we use the Skorokhod topology $J_1$  and general
methods of approximations of stochastic integrals and solutions of
SDEs developed in \cite{jmp,ms,s1,s2}.
For the convenience of the reader we prove in Appendix a general tightness criterion and a functional limit theorem for sequences of integrals with respect to c\`adl\`ag functions  with bounded $p$-variation.

In Section 4 we apply our deterministic results  to obtain the
existence, uniqueness and approximation of solutions  to SDEs of the form
(\ref{eq1.1}). In particular,  we show that if
$f,g$ satisfy (H1) and (H2) then (\ref{eq1.1}) has a unique strong
solution $(X,K)$.   Moreover,  we show convergence  to $(X,K)$ of some easy to implement  approximations
$(X^n,K^n)$ constructed in analogy with the classical Euler
scheme.
To illustrate  how our results work in practice, at the end of the
paper we consider fractional SDEs with constraints of the form
 \begin{equation}\label{eq1.4}
 X_t = X_0 + \int_0^t f(X_{s-})\,da_s+\int_0^t
g(X_{s-})\,dZ^{H}_s+K_t, \quad t\in\mathbb R^+.
\end{equation}
Here $a:\RRp\to\RR$ is a continuous function with locally bounded
variation and
$
Z^{H,i}=\int_0^\cdot\sigma^i_s\,d B^{H,i}_s$, $t\in\Rp$,
where $B^{H,1},...,B^{H,d}$  are independent fractional Bro\-wnian
motions and $\sigma^i:\RRp\to\RR$ are such that
$\|\sigma^i\|_{\LH{[0,T]}}:=(\int_0^T|\sigma^i_s|^{1/H}ds)^H<\infty$,
$T>0$, $i=1,\dots,d$. Under the last assumption $\ZH$ is a
centered Gaussian process with continuous trajectories  such that
$P(V_p(\ZH)_T<\infty)=1$, $p>1/H$, $T\in\Rp$ (see Section 4), so
(\ref{eq1.4}) is a particular case of (\ref{eq1.1}). Of course,
(\ref{eq1.4}) generalizes classical fractional SDEs driven by
$B^H$.

 In the sequel  we will use the  following  notation.
$\Md$ is the space of $d\times d$ real matrices
$A$, with the matrix norm $\|A\|=\sup\{|Au|;u\in\Rd,|u|=1\}$,
where $|\cdot|$ denotes  the usual Euclidean norm in $\Rd$, $\RRp=[0,\infty)$. $\Dd$
is the space of c\`adl\`ag  mappings $x:\RRp\to\RR^d$, i.e.
mappings which are right continuous and admit left-hands limits
 equipped
with the Skorokhod topology  $J_1$. For $x\in\Dd$, $t>0$ we denote
$x_{t-}=\lim_{s\uparrow t}x_s$ and $v_{p}(x)_{[a,b]} = \sup_\pi
\sum_{i=1}^n |x_{t_i}-x_{t_{i-1}}|^p <\infty$, where the supremum
is taken over all subdivisions $\pi=\{a=t_0<\ldots<t_n=b\}$ of
$[a,b]$. $V_p(x)_{[a,b]}=(v_p(x)_{[a,b]})^{1/p}$ and $\bar
V_p(x)_{[a,b]}=V_p(x)_{[a,b]}+|x_a|$ is the usual variation norm.
For simplicity of notation we write
$v_p(x)_T=v_p(x)_{[0,T]}$, $V_p(x)_T= V_p(x)_{[0,T]}$ and $\bar
V_p(x)_T= \bar V_p(x)_{[0,T]}$. If $x\in\DMd$  then in the
definition of $p$-variation $v_p$ we use the matrix norm
$\|\cdot\|$ in place of the Euclidean norm.
We write $x\leq x'$,
$x,x'\in\Dd$ if $x^i_t\leq x'^i_t$, $t\in\Rp$, $i=1,\dots,d$.
Every process $Y$ appearing in the sequel is assumed to have
c\`adl\`ag trajectories.

\section{Lipschitz continuity of the solution of the Skorokhod
problem in $p$-variation norm}

Let $y,l\in\Dd$ be such that $l_0\leq y_0$. We recall that a~pair
$(x,k)\in\Ddd$ is called a solution of the  Skorokhod problem
associated with $y$ and lower barrier $l$ ($(x, k)=SP_{l}(y)$ for
short) if
\begin{description}
\item[(i)] $ x_t = y_t+k_t\geq l_t$, $t\in\Rp$,
\item[(ii)] $k_0=0$, $k=(k^1,\dots,k^d)$, where  $k^i$ are
nondecreasing functions  such that for every $t\in\Rp$,
\[
\int_0^t(x^i_s-l^i_s)\,dk^i_s=0,\quad i=1,\dots,d.
\]
\end{description}

The Lipschitz continuity  of the mapping $(y,l)\mapsto(x,k)$ in
the supremum norm is well known. More precisely, let $(x,k)=SP_l(y)$,
$(x',k')=SP_{l'}(y')$. Since $k_t=\sup_{s\leq t}(y_s-l_s)^-$ and
$k'_t=\sup_{s\leq t}(y'_s-l'_s)^-$, for any $T\in\Rp$ we have
\begin{equation}\label{eq2.1}
\sup_{t\leq T} |x_t-x'_t|\leq2\sup_{t\leq T}|y_t-y'_t|
+\sup_{t\leq T}|l_t-l'_t|
\end{equation}
and
\begin{equation}\label{eq2.2}
 \sup_{t\leq T}|k_t-k'_t|\leq\sup_{t\leq T}|y_t-y'_t|
+\sup_{t\leq T}|l_t-l'_t|.\end{equation} On the other hand,  it
was observed in Ferrante and Rovira \cite{fr1} that above property
does not hold in $\lambda$-H\"older norm. We  will show that the
Lipschitz continuity  of the mapping $(y,l)\mapsto(x,k)$  holds in
the variation norm. A key step in proving it is the following
estimate.

\begin{theorem}\label{thm1}
For any  $y^1,y^2\in\D$ and  $T\in\Rp$,
\[
v_p(\sup_{s\leq\cdot}y^1_s-\sup_{s\leq\cdot}y^2_s)_T\leq v_p(y^1-y^2)_T.
\]
\end{theorem}
\begin{proof} It is clear that  without loss of generality we may and will assume that
\begin{equation}\label{eq2.0}
v_p(\sup_{s\leq\cdot}y^1_s-\sup_{s\leq\cdot}y^2_s)_T> 0.
\end{equation}
 {\em Step 1.} We assume additionally that  $y^1,y^2$
are step functions of the form
\[
y^j_t=y_{j,i},\quad t\in[t_{i-1},t_{i}),\quad i=1,\dots,n-1
\]
and $y^j_t=y_{j,n}$, $t\in[t_{n-1},t_n=T]$, $j=1,2$, for some
partition $0=t_0<t_1<\dots<t_n=T$ of the interval $[0,T]$.

Set $Y^j_{k}=\max_{1\leq i\leq k}y_{j,i}$, $k=1,\dots,n$, $j=1,2$. By (\ref{eq2.0}) it is clear that there exists $k$
such that $Y^1_{k}>Y^1_{k-1}$ or $Y^2_{k}>Y^2_{k-1}$. Without loss of generality we  will assume that for any
$k=2,\dots,n$,
\begin{equation}\label{eq2.3}
Y^1_{k}>Y^1_{k-1} \quad\mathrm{or}\quad
Y^2_{k}>Y^2_{k-1}.
\end{equation}
Indeed, if (\ref{eq2.3}) does not hold then  we set $u_0=0$,
\[
u_k=\inf\{i>u_{k-1};Y^1_{i}>Y^1_{i-1}\,\mathrm{or}\,Y^2_{i}>Y^2_{i-1}\}\wedge n,\quad
k=1,\dots,n
\]
and $\tilde n=\inf\{k;u_k=n\}$,  $\tilde y^j_t=y_{j,u_k}$,
$t\in[t_{u_{k-1}},t_{u_{k}})$ for $k=1,\dots,\tilde n-1$, $\tilde
y^j_t=y_{j,\tilde n}$ for $t\in[t_{u_{\tilde n-1}},t_{u_{\tilde
n}}=T]$, $j=1,2$. Then (\ref{eq2.3}) holds true for the functions
$\tilde y^1,\tilde y^2$ and
$v_p(\sup_{s\leq\cdot}y^1_s-\sup_{s\leq\cdot}y^2_s)_T
=v_p(\sup_{s\leq\cdot}\tilde y^1_s-\sup_{s\leq\cdot}\tilde
y^2_s)_T$,  $v_p(\tilde y^1-\tilde y^2)_T\leq v_p(y^1-y^2)_T$.

It
is clear that there exist numbers $0=i_0<i_1<\ldots<i_m=n$ such
that
\begin{equation}\label{eq2.4}
v_p(\sup_{s\leq\cdot}y^1_s-\sup_{s\leq\cdot}y^2_s)_T
=\s_{k=1}^{m}{|(Y^1_{i_k}-Y^1_{i_{k-1}})
-(Y^2_{i_k}-Y^2_{i_{k-1}})|^p}
\end{equation}
and
\begin{equation}\label{eq2.5}
(Y^1_{i_k}-Y^1_{i_{k-1}})-(Y^2_{i_k}-Y^2_{i_{k-1}})\neq 0
\end{equation}

for  $k=1,\dots,m$. In particular, this implies that  if $m\geq 2$ then for  $k=2,\dots,m$
we have
\begin{equation}\label{eq2.6}
\big((Y^1_{i_{k-1}}-Y^1_{i_{k-2}})-(Y^2_{i_{k-1}}-Y^2_{i_{k-2}})\big)
\big((Y^1_{i_k}-Y^1_{i_{k-1}})-(Y^2_{i_k}-Y^2_{i_{k-1}})\big)<0.
\end{equation}
Indeed, if (\ref{eq2.6}) is not satisfied then by (\ref{eq2.5}),
\begin{align*}
&|(Y^1_{i_{k-1}}-Y^1_{i_{k-2}})-(Y^2_{i_{k-1}}-Y^2_{i_{k-2}})|^p
+|(Y^1_{i_k}-Y^1_{i_{k-1}})-(Y^2_{i_k}-Y^2_{i_{k-1}})|^p\\
&\quad< |(Y^1_{i_{k-1}}-Y^1_{i_{k-2}})-(Y^2_{i_{k-1}}-Y^2_{i_{k-2}})
+ (Y^1_{i_k}-Y^1_{i_{k-1}})-(Y^2_{i_k}-Y^2_{i_{k-1}})|^p\\
&\quad= |(Y^1_{i_k}-Y^1_{i_{k-2}})-(Y^2_{i_k}-Y^2_{i_{k-2}})|^p,
\end{align*}
which contradicts (\ref{eq2.4}). Set $l_k^j=\max\{i\mr
i_k:y^j_i=Y^j_i\}$, $j=1,2$,  and
$l_k^\wedge=\min\{l_k^1,l_k^2\}$, $l_k^\vee=\max\{l_k^1,l_k^2\}$,
$k=1,\dots,m$. Then
\begin{equation}\label{eq2.7}
y_{1,l_k^1}=Y^1_{l_k^1}=Y^1_{l_k^1+1}=\ldots=Y^1_{i_k}\mbox{ and }
y_{2,l_k^2}=Y^2_{l_k^2}=Y^2_{l_k^2+1}=\ldots=Y^2_{i_k}.
\end{equation}
We claim that for any  $k=1,\dots,m$,
\begin{equation}\label{eq2.8}
i_{k-1}\leq l^\wedge_k\leq l^\vee_k\leq i_k.
\end{equation}
The  last two inequalities are obvious. Moreover, $0=i_0\leq
l^\wedge_1$. Assume that there exists $2\le k\le m$ such that
$i_{k-1}>l^\wedge_k$. In what follows we will only  consider the
case $l^\wedge_k=l^1_k$ (the case $l^\wedge_k=l^2_k$ can be
handled in much the same way). We have $i_{k-2}<l^\wedge_k$,
because if $i_{k-2}\geq l^\wedge_k$   then by (\ref{eq2.7}),
\begin{align*}
|(Y^1_{i_{k-1}}-Y^1_{i_{k-2}})-(Y^2_{i_{k-1}}-&Y^2_{i_{k-2}})|^p
+|(Y^1_{i_k}-Y^1_{i_{k-1}})- (Y^2_{i_k}-Y^2_{i_{k-1}})|^p\\
&\quad=|(Y^2_{i_{k-1}}-Y^2_{i_{k-2}})|^p+|(Y^2_{i_k}-Y^2_{i_{k-1}})|^p\\
&\quad<|(Y^2_{i_k}-Y^2_{i_{k-2}})|^p\\
&\quad=|(Y^1_{i_k}-Y^1_{i_{k-2}})-(Y^2_{i_k}-Y^2_{i_{k-2}})|^p,
\end{align*}
which contradicts (\ref{eq2.4}). From the inequality
$i_{k-2}<l^\wedge_k$ and (\ref{eq2.3}) it follows that
\begin{equation}
\label{eq2.9}
Y^1_{i_{k-2}}\leq Y^1_{l_k^1}=Y^1_{l_k^1+1}
=\ldots=Y^1_{i_{k-1}}=\ldots=Y^1_{i_k}
\end{equation}
and
\begin{equation}\label{eq2.10}
Y^2_{i_{k-2}}\leq
Y^2_{l_k^1}<Y^2_{l_k^1+1}<\ldots<Y^2_{i_{k-1}}<\ldots<Y^2_{i_k}.
\end{equation}
Since \[(Y^1_{i_k}-Y^1_{i_{k-1}})-(Y^2_{i_k}-Y^2_{i_{k-1}})
=-(Y^2_{i_k}-Y^2_{i_{k-1}})<0,\]  (\ref{eq2.6}), (\ref{eq2.9}) and
(\ref{eq2.10}) imply that
\begin{align*}
0<(Y^1_{i_{k-1}}-Y^1_{i_{k-2}})-(Y^2_{i_{k-1}}-Y^2_{i_{k-2}}) &=
(Y^1_{l^1_k}-Y^1_{i_{k-2}})-(Y^2_{i_{k-1}}-Y^2_{i_{k-2}}) \\
&<(Y^1_{l^1_k}-Y^1_{i_{k-2}})-(Y^2_{l^1_k}-Y^2_{i_{k-2}}),
\end{align*}
hence that
\begin{equation}\label{eq2.11}
|(Y^1_{i_{k-1}}-Y^1_{i_{k-2}})-(Y^2_{i_{k-1}}-Y^2_{i_{k-2}})|^p
<|(Y^1_{l^1_k}-Y^1_{i_{k-2}})-(Y^2_{l^1_k}-Y^2_{i_{k-2}})|^p.
\end{equation}
Similarly,
\begin{align*}
0<-\big((Y^1_{i_k}-Y^1_{i_{k-1}})-(Y^2_{i_k}-Y^2_{i_{k-1}})\big)&
=(Y^2_{i_k}-Y^2_{i_{k-1}})\\
&<(Y^2_{i_k}-Y^2_{l^1_k})\\
&=(Y^1_{i_k}-Y^1_{l^1_k})-(Y^2_{i_k}-Y^2_{l^1_k}),
\end{align*}
which implies that
\begin{equation}\label{eq2.12}
|(Y^1_{i_k}-Y^1_{i_{k-1}})-(Y^2_{i_k}-Y^2_{i_{k-1}})|^p
<|(Y^1_{i_k}-Y^1_{l^1_k})-(Y^2_{i_k}-Y^2_{l^1_k})|^p.
\end{equation}
Combining (\ref{eq2.11}) with (\ref{eq2.12}) we obtain
\begin{align*}
&|(Y^1_{i_{k-1}}-Y^1_{i_{k-2}})-(Y^2_{i_{k-1}}-Y^2_{i_{k-2}})|^p
+|(Y^1_{i_k}-Y^1_{i_{k-1}})-(Y^2_{i_k}-Y^2_{i_{k-1}})|^p\\
&\quad<|(Y^1_{l^1_k}-Y^1_{i_{k-2}})-(Y^2_{l^1_k}-Y^2_{i_{k-2}})|^p
+|(Y^1_{i_k}-Y^1_{l^1_k})-(Y^2_{i_k}-Y^2_{l^1_k})|^p.
\end{align*}
which contradicts (\ref{eq2.4}) and completes the proof of
(\ref{eq2.8}). It is clear that for any  $k$  we have $
y_{1,l_k^1}\wlr y_{1,l_k^2}$ and $y_{2,l_k^2}\wlr y_{2,l_k^1}$.
Consequently, in the case
$Y^1_{i_k}-Y^1_{i_{k-1}}>Y^2_{i_k}-Y^2_{i_{k-1}}$ we have
\begin{align*}
0<(Y^1_{i_k}-Y^1_{i_{k-1}})-(Y^2_{i_k}-Y^2_{i_{k-1}})&
=(y_{1,l_k^1}-y_{1,l^1_{k-1}})-(y_{2,l^2_k}-y_{2,l^2_{k-1}})\\
&\mr(y_{1,l_k^1}-y_{1,l^2_{k-1}})-(y_{2,l^1_k}-y_{2,l^2_{k-1}}).
\end{align*}
Hence
\begin{equation}\label{eq2.13}
\!\!|(Y^1_{i_k}-Y^1_{i_{k-1}})-(Y^2_{i_k}-Y^2_{i_{k-1}})|^p
\mr|(y_{1,l_k^1}-y_{1,l^2_{k-1}})-(y_{2,l^1_k}-y_{2,l^2_{k-1}})|^p.
\end{equation}
Similarly one can check that
if $Y^1_{i_k}-Y^1_{i_{k-1}}<Y^2_{i_k}-Y^2_{i_{k-1}}$ then
\begin{equation}\label{eq2.14}
\!\!|(Y^1_{i_k}-Y^1_{i_{k-1}})-(Y^2_{i_k}-Y^2_{i_{k-1}})|^p\mr
|(y_{1,l_k^2}-y_{1,l_{k-1}^1})-(y_{2,l^2_k}-y_{2,l^1_{k-1}})|^p.
\end{equation}
By (\ref{eq2.13}) and (\ref{eq2.14}),
\[
\s_{k=1}^{m}|(Y^1_{i_k}-Y^1_{i_{k-1}})-(Y^2_{i_k}-Y^2_{i_{k-1}})|^p
\leq
\s_{k=1}^{m}|(y_{1,l_k}-y_{1,l_{k-1}})-(y_{2,l_k}-y_{2,l_{k-1}})|^p,
\]
where $l_k=l_k^1$ or $l_k=l_k^2$. Moreover, by (\ref{eq2.8}),
$i_{k-1}\leq l_k\leq i_k$ for $k=1,\dots,m$. Hence
\[
v_p(\sup_{s\leq\cdot}y^1_s-\sup_{s\leq\cdot}y^2_s)_T\leq
\s_{k=1}^{m}|(y^1_{t_{l_k}}-y^1_{t_{l_{k-1}}})
-(y^2_{t_{l_k}}-y^2_{t_{l_{k-1}}})|^p\] for some partition
$0=t_{l_0}<t_{l_1}<\ldots<t_{l_m}\leq T$, which proves the theorem
under our additional assumption.

{\em Step 2.} The general case.

Let $\{y^{1,n}\}$ and $\{y^{2,n}\}$ be sequences of
discretizations of $y^1$  and $y^2$, respectively, i.e.
$y^{1,n}_t=y^1_{k/n}$, $y^{2,n}_t=y^{2}_{k/n}$,
$t\in[k/n,(k+1)/n)$, $k\in\No$. By Step 1, for~any~$n\in\N$ and
$T\in\Rp$ we have
\[v_p(\sup_{s\leq\cdot}y^{1,n}_s-\sup_{s\leq\cdot}y^{2,n}_s)_T\leq
v_p (y^{1,n}-y^{2,n})_T.\] Clearly, $v_p(y^{1,n}-y^{2,n})_T\leq
v_p(y^1-y^2)_T$, $n\in\N$, $T\in\Rp$. By using e.g. \cite[Chapter 3, Proposition 6.5]{ek} one can check that
\[y^{1,n}\lra y^1\,\, {\rm and}\,\, y^{2,n}\lra y^2\,\, {\rm in}\,\,\D.\]
 Hence and by \cite[Chapter VI. Proposition 2.2]{js}
\[(y^{1,n},y^{2,n})\lra(y^1,y^2)\,\, {\rm in}\,\,\DD,\] which together with \cite[Chapter VI. Proposition 2.4]{js} implies that
\[
\sup_{s\leq\cdot}y^{1,n}_s-\sup_{s\leq\cdot}y^{2,n}_s
\lra\sup_{s\leq\cdot}y^1_s-\sup_{s\leq\cdot}y^2_s \quad\mbox{\rm in}\,\,\D.
\]
Therefore for any $T$ such that  $\Delta y^1_T=\Delta y^2_T=0$,
\begin{align*}
v_p(\sup_{s\leq\cdot}y^1_s-\sup_{s\leq\cdot}y^2_s)_T
&\leq\liminf_{n\to\infty}
v_p(\sup_{s\leq\cdot}y^{1,n}_s-\sup_{s\leq\cdot}y^{2,n}_s)_T\\
&\leq\sup_n v_p(y^{1,n}-y^{2,n})_T\leq v_p(y^1-y^2)_T.
\end{align*}
If $\Delta y^1_T\neq0$ or $ \Delta y^2_T\neq0$ then there exists a
sequence~$\{T_k\}$ such that $T_k\downarrow T$ and $\Delta
y^1_{T_k}=\Delta y^2_{T_k}=0$, $k\in\N$. Then
$v_p(\sup_{s\leq\cdot}y^1_s-\sup_{s\leq\cdot}y^2_s)_{T_k}\leq
v_p(y^1-y^2)_{T_k}$, $k\in\N$, so letting $k\to\infty$ we obtain
the desired result.
\end{proof}

\begin{theorem}\label{thm2}
Assume  $l,y,l',y'\in\Dd$ are such that $l_0\leq y_0$  and
$l'_0\leq y'_0$. Let $(x, k)=SP_l(y)$ and $(x', k')=SP_{l'}(y')$.
Then for any $T\in\Rp$,
\[
V_p(x-x')_T \leq (d+1) V_p( y-y')_T+d| y_0-y'_0|+
dV_p(l-l')_T+d|l_0-l'_0|\] and \[ V_p(k-k')_T \leq d
V_p(y-y')_T+d|y_0-y'_0|+ dV_p(l-l')_T+d|l_0-l'_0|.
 \]
\end{theorem}
\begin{proof}
Observe that $k_t=\sup_{s\leq t}(y_s-l_s)^-$
$=\sup_{ s\leq 1+t}\bar y_s$, where $\bar y_s=0$ for  $s\in[0,1)$
and $\bar y_s=l_{s-1}-y_{s-1}$  for $s\geq1$.  Similarly,
$k'_t=\sup_{ s\leq 1+t}\bar y'_s$, where $\bar y'_s=0$ for
$s\in[0,1)$ and $\bar y'_s=l'_{s-1}-y'_{s-1}$  for $s\geq1$. By
Theorem \ref{thm1},
\begin{align*}
V_p(k-k')_T&=V_p(\sup_{s\leq\cdot}\bar y_s-\sup_{s\leq\cdot}\bar y'_s)_{T+1}\\
&\leq d^{(p-1)/p}(\sum_{i=1}^dv_p(\sup_{s\leq\cdot}\bar y^i_s
-\sup_{s\leq\cdot}\bar y'^i_s)_{T+1})^{1/p}\\
&\leq d^{(p-1)/p}(\sum_{i=1}^dv_p(  \bar y^i-\bar y'^i)_{T+1})^{1/p}\\
&\leq d\max_i V_p(\bar y^i-\bar y'^i)_{T+1} \leq dV_p(\bar y-\bar
y')_{[0,T+1]}.
\end{align*}
Since
\begin{align*}
V_p(\bar y-\bar y')_{T+1}&\leq V_p(\bar y-\bar y')_{[0,1]}
+ V_p(\bar y-\bar y')_{[1,T+1]}\\
&= |(y_0-y'_0)-(l_0-l'_0)|+V_p((y-y')-(l-l'))_T\\
&\leq V_p(y-y')_T+|y_0-y'_0|+ V_p(l-l')_T+|l_0-l'_0|
\end{align*}
and
\[
V_p(x-x')_T \leq V_p(y-y')_T+ V_p( k-k')_T,
\]
the proof is complete.
\end{proof}

\begin{corollary} \label{cor1}
Under the assumptions of  Theorem \ref{thm2}, for every $T\in\Rp$
the estimates (\ref{eq1.2}), (\ref{eq1.3}) hold true.
\end{corollary}
\begin{proof} It suffices to observe that $x_0=y_0$,
$x'_0=y'_0$ and $k_0=k'_0=0$.
\end{proof}

\begin{remark}\label{rem1}
(a) The case $p=d=1$ was studied earlier in \cite{sw} (see
also \cite{sa}).

(b) Let $(x^n,k^n)=SP_{l^n}(y^n)$, $(x,k)=SP_l(y)$. By
(\ref{eq2.1})  and (\ref{eq2.2}) it is clear  that if $(y^n,l^n)$
tends to $(y,l)$ in the uniform norm then  $(x^n,k^n)$ tends to
$(x,k)$ in the uniform norm. From this one can deduce that
\[
(y^n,l^n)\rightarrow(y,l)\,\,\mbox{\rm in}\,\,
\Ddd\Rightarrow(x^n,k^n,y^n,l^n)
\rightarrow(x,k,y,l)\,\,\mbox{\rm in}\,\,\Dcd
\]
and if $\{(y^n,l^n)\}$ is relatively compact in $\Ddd$ then
$\{(x^n,k^n,y^n,l^n)\}$ is relatively compact in $\Dcd$ (see e.g.
\cite{sw}). From  Corollary \ref{cor1} it also follows  that if
$(y^n,l^n)$ tends to $(y,l)$ in  the variation norm then
$(x^n,k^n)$ tends to $(x,k)$ in the variation norm.
\end{remark}
\begin{remark}\label{cor2}
 Let $(x, k)=SP_l(y)$. Since $k_t=\sup_{s\leq t}(y_s-l_s)^-$, for any $T\in\Rp$ we have
\[
\bar V_p(k)_T\leq d\sup_{t\leq T}|y_t| +d\sup_{t\leq T}|l_t|.
\]
and
\[
\bar V_p(x)_T\leq(d+1)\bar V_p(y)_T +d\sup_{t\leq T}|l_t|.\]

\end{remark}

\section{Deterministic integral equations}

Let  $x\in\DMd$, $z\in\Dd$  be such that  $V_q(x)_T<\infty$,
$V_p(z)_T<\infty$, $T\in\Rp$, where $1/p+1/q>1$, $p,q\geq1$. It
is well known  (see, e.g., \cite{du,dn,dn1,y} ) that the Riemann-Stiltjes integral
$\int_0^\cdot x_{s-} dz_s$ is a well defined c\`adl\`ag function
such that for any $a<b$,
\begin{equation}\label{eq3.2}
V_p(\int_a^\cdot x_{s-} dz_s)_{[a,b]}\leq C_{p,q}\bar
V_q(x)_{[a,b)}V_p(z)_{[a,b]},
\end{equation}
where $C_{p,q}=\zeta(p^{-1}+q^{-1})$ and $\zeta$ denotes the
Riemann zeta function, i.e. $\zeta(x)=\sum_{n=1}^\infty 1/n^x$.

Let $a\in\D$, $z,l\in\Dd$ be such that $V_1(a)_T$, $V_p(z)_T$,
 $T\in\Rp$, and $x_0\geq l_0$. We  consider
equations with constraints of the form
\begin{equation}\label{eq3.1}
x_t=x_0+\int_0^t f(x_{s-})\,da_s+\int_0^tg(x_{s-})\,dz_s
+k_t,\quad t\in\Rp,
\end{equation}
where $f:\Rd\rightarrow\Rd$  and $g:\Rd\rightarrow\Md$ are given
functions and the integral with respect to $z$ is a Riemann-Stieltjes integral.
\begin{definition}\label{def}
 We say that a~pair $(x,k)\in\Ddd$ is a solution of
(\ref{eq3.1}) if $V_p(x)_T<\infty$, $T\in\Rp$,  and
$(x,k)=SP_l(y)$, where
\[
y_t=x_0+\int_0^t f(x_{s-})\,da_s+\int_0^tg(x_{s-})\,dz_s,\quad t\in\Rp.
\]

\end{definition}

We will need the following conditions.
\begin{enumerate}
\item[(H1)] (a) $f:\Rd\rightarrow\Rd$  is continuous and
satisfies the linear  growth condition, i.e. there is $L>0$ such that
\[
|f(x)|\leq L(1+|x|),\quad x\in\Rd.
\]

(b)  $g:\Rd\rightarrow\Md$  is H\"older continuous function of
order   $\alpha\in(p-1,1]$, i.e. there is $C_{\alpha}>0$ such that
\[||g(x)-g(y)||\leq C_\alpha |x-y|^\alpha,\quad
x,y\in\Rd.\]
\item[(H2)] (a) $f:\Rd\rightarrow\Rd$  is locally Lipschitz
continuous, i.e. for any $k\in\N$ there is $L_k>0$ such that
\[|f(x)-f(y)|\leq L_k|x-y|,\quad |x|,|y| \le k.\]

(b) $g:\Rd\rightarrow\Md$, its each component $g_{i,j}$ is
differentiable and  there are $\gamma\in(p-1,1]$ and
$C_{k,\gamma}>0$ such that for every $k\in\N$,
\[
|\nabla_xg_{i,j}(x)-\nabla_xg_{i,j}(y)|\leq C_{k,\gamma} |x-y|^\gamma,\quad
|x|,|y| \le k,\,\, i,j=1,\dots,d.
\]
\end{enumerate}
Similar sets of conditions were considered in papers on
equations without constraints driven by functions (processes) with
bounded $p$-variation (see, e.g., \cite{du,fsz,k1,k2,ly,nr,ru}).

The outline of the rest of Section 3 is as follows. First we study
convergence of solutions of equations of the type (\ref{eq3.1}) in
the Skorokhod topology $J_1$. In our proofs we use a general
tightness criterion and a functional limit theorem for  sequences of
integrals with respect to c\`adl\`ag functions (see Appendix). As a
simple corollary to our convergence result we show that under (H1)
there exists a solution of (\ref{eq3.1}).  Next, assuming
additionally (H2), we prove that (\ref{eq3.1}) has a unique
solution $(x,k)$. Under (H1)  and (H2), we show at the end of Section 3,
 that $(x,k)$ can by approximated by simple  and easy to
implement approximation schemes.

Let $z^n,l^n\in\Dd$, $a^n\in\D$ be such that $x^n_0\geq l^n_0$  and
$V_1(a^n)_T$, $V_p(z^n)_T$ $<\infty$, $T\in\Rp$.
We will consider solutions $(x^n,k^n)\in\Ddd$ of equations with
constraints of the form
\begin{equation}\label{eq3.11}
x^n_t=x^n_0+\int_0^t
f(x^n_{s-})\,da^n_s+\int_0^tg(x^n_{s-})\,dz^n_s +k^n_t,\quad
t\in\Rp,
\end{equation}
i.e.  $(x^n,k^n)=SP_{l^n}(x^n_0+\int_0^\cdot
f(x^n_{s-})\,da^n_s+\int_0^\cdot g(x^n_{s-})\,dz^n_s)$ and
$V_p(x^n)_T<\infty$, $T\in\Rp$.

\begin{theorem}\label{thm3}
Suppose that functions $f$ and $g$ satisfy \mbox{\rm(H1)}. Let $\{a^n\}\subset \D$,
$\{z^n\}$, $\{l^n\}$ $\subset\Dd$ be sequences such that
$\sup_n V_1(a^n)_T<\infty$, $\sup_n V_p(z^n)_T<\infty$, $T\in\Rp$ and
\[(\,x^n_0,a^n,\,z^n,\,l^n\,)\longrightarrow\,(\,x_0,a,\,z\,,\,l\,)\,\, in\,\,\Rd\times\Dddt.\]
If $\{(x^n,k^n)\}$ is a sequence of solutions of (\ref{eq3.11})  then
\[
\{(x^n\,,k^n)\}\quad is \,\,relatively\,\, compact\,\, in\,\,\Ddd
\]
and its every limit point is a solution of (\ref{eq3.1}).
\end{theorem}
\begin{proof}
The lemma below will be our main tool in the proof.

\begin{lemma}\label{lem}
Assume $f$ and $g$ satisfy \mbox{\rm(H1)}.  Let $(x,k)$ be a solution of
(\ref{eq3.1}) and let $b,T>0$. If
\[\max(V_1(a)_T,V_p(z)_T,\sup_{t\leq T}|l_t|)\leq b\] then there is  $\bar
C=C(d,p,\alpha,L,g(0),x_0,b)>0$ such that $\bar V_p(x)_T\leq \bar
C$.
\end{lemma}
\begin{proof}
By Remark \ref{cor2}, for any $t\leq T$,
\begin{align*}
\bar V_{p}(x)_t&\leq (d+1)\bar V_p(y)_t+d\sup_{s\leq t}|l_s|\\
& \leq (d+1)\left[|x_0|+V_p(\int_0^\cdot
f(x_{s-})\,da_s)_t+V_p(\int_0^\cdot
g(x_{s-})\,dz_s)_t\right]+d\sup_{s\leq t}|l_s|.
\end{align*}
We have
\[
V_p(\int_0^\cdot f(x_{s-})\,da_s)_t\leq V_1(a)_t\,\sup_{s\leq
t}|f(x_{s-})| \leq LV_1(a)_t\,(1+\bar V_p(x)_t)
\]
and, by (\ref{eq3.2}),
\begin{align*}
V_p(\int_0^\cdot g(x_{s-})\,dz_s)_t&\leq
C_{p,p/\alpha}\bar V_{p/\alpha}(g(x))_tV_p(z)_t\\
&\leq C_{p,p/\alpha}(C_{\alpha} V^\alpha_{p}(x)_t+C_{\alpha}|x_0|^\alpha+|g(0)|)V_p(z)_t\\
&\leq C_{p,p/\alpha}[C_{\alpha}(\alpha\bar V_{p}(x)_t+2(1-\alpha))+|g(0)|]V_p(z)_t\\
&\leq DV_p(z)_t(1+\bar V_p(x)_t),
\end{align*}
where $D=C_{p,p/\alpha}(C_\alpha(2-\alpha)+|g(0)|)$.

Set $t_1=\inf\{t;LV_1(a)_t>\frac1{4(d+1)}\,\mbox{ or }
DV_p(z)_t>\frac1{4(d+1)}\}\wedge T$.
By the above,
\[
\bar V_{p}(x)_{[0,t_1)} \leq (d+1)|x_0|
+\frac12 (1+\bar V_p(x)_{[0,t_1)})+d\sup_{s\leq t_1}|l_s|,
\]
which implies that $\bar V_{p}(x)_{[0,t_1)} \leq
2(d+1)|x_0|+1+2d\sup_{s\leq t_1}|l_s|$. Since
\begin{align*}
|\Delta x_{t_1}|&\leq |f(x_{t_1-})\Delta a_{t_1}|
+|g(x_{t_1-})\Delta z_{t_1}|+|\Delta l_{t_1}|\\
&\leq \big(L(1+|x_{t_1-}|)+C_\alpha|x_{t_1-}|^\alpha+|g(0)|+2\big)b,
\end{align*}
there exist $ C_1,C_2>0$ depending only on $d,p,\alpha,L,g(0),b$ such that
\[
\bar V_{p}(x)_{[0,t_1]} \leq C_1+C_2|x_0|.
\]
Set
$t_k=\inf\{t>t_{k-1};LV_1(a)_{[t_{k-1},t]}>\frac1{4(d+1)} \mbox{
or } DV_p(z)_{[t_{k-1},t]}>\frac1{4(d+1)}\}\wedge T$, $k=2,3,\dots$,
and observe that for the same constants $C_1,C_2$,
\[
V_{p}(x)_{[t_{k-1},t_k]} \leq C_1+C_2|x_{t_{k-1}}|
\leq C_1+C_2\bar V_p(x)_{[0,t_{k-1}]}.
\]
What is left  is to show  that $m=\inf\{k;t_k=T\}$ is finite and
depends only on $p,\alpha,L,g(0),b$. To see this, let us observe that
\begin{align*}
m\left(\frac{1}{4(d+1)}\right)^p&\leq \sum_{k=1}^m
LV_1(a)_{[t_{k-1},t_k]}
+D^p v_p(z)_{[t_{k-1},t_k]}\\
&\leq Lb+D^p b^p,
\end{align*}
which  yields $m\leq(4(d+1))^p[Lb+D^pb^p]$. This completes the proof.
\end{proof}
\medskip

By Lemma \ref{lem} $\sup_n \bar V_p(x^n)_T<\infty$.
Since $\bar V_{p/\alpha}(g(x^n))_T\leq
C_{\alpha}V^\alpha_p(x^n)_T+|g(x^n_0)|$, we also have
\begin{equation}\label{eq3.12}
\sup_n \bar V_{p/\alpha}(g(x^n))_T<\infty,\quad T\in\Rp.
\end{equation}
Since $f$ is continuous, there exists a sequence of Lipschitz
continuous functions $\{f^k\}$ such that for any compact
$K\subset\Rd$,  $\sup_{u\in K}|f^k(u)-f(u)|\rightarrow0$. Hence and from the fact that
$\sup_n \bar V_p(x^n)_T<\infty$, $T\in\Rp$, it follows  that
\begin{equation}
\label{eq3.13} \sup_n \bar V_p(f^k(x^n))_T<\infty, \quad
T\in\Rp,\,\, k\in\N
\end{equation}
and
\begin{equation}
\label{eq3.14} \lim_{k\to\infty}\limsup_{n\to\infty}\sup_{t\leq
T}|f^k(x^n_t)-f(x_t)|=0, \quad T\in\Rp.\end{equation}  Putting $q=p/\alpha$ in Corollary \ref{cora1} and using (\ref{eq3.12}), (\ref{eq3.13}) we show that for any
$k\in\N$,
\[
\{(\int_0^\cdot f^k(x^n_{s-})\,da^n_s,a^n,
\int_0^\cdot g(x^n_{s-})\,dz^n_s,z^n)\}\,
\mbox{\rm is relatively compact in}\,\Ddddt.
\]
By the above and (\ref{eq3.14}),
\[
\{(\int_0^\cdot f(x^n_{s-})\,da^n_s,a^n,
\int_0^\cdot g(x^n_{s-})\,dz^n_s,z^n)\}\,
\mbox{\rm is relatively compact in}\,\Ddddt.
\]
Therefore, $\{(a^n,y^n,z^n,l^n)\}$, with
$y^n=x^n_0+\int_0^\cdot f(x^n_{s-})\,da^n_s+\int_0^\cdot
g(x^n_{s-})\,dz^n_s$,
  is relatively
compact in $\Ddddt$. Now, put $(x^n,k^n)=SP_{l^n}(y^n)$, $n\in\N$,
and observe that by Remark \ref{rem1}(b),
\[
\{(x^n,a^n,z^n,l^n)\}\quad\mbox{\rm is relatively compact
in}\,\,\Ddddt.
\]
Assume that $(x^n,\,a^n,\,z^n,\,l^n)\rightarrow(x,\,a,\,z,\,l)$ in
$\Ddddt$. By  Corollary \ref{cora2}, $(y^n,l^n)\rightarrow(y,l)$,
where $y=x_0+\int_0^\cdot f(x_{s-})\,da_s+\int_0^\cdot
g(x_{s-})\,dz_s$. Consequently, by Remark \ref{rem1}(a),
\[(x^n,k^n)=SP_{l^n}(y^n)\rightarrow SP_l(y)=(x,k)\,\,
\text{in}\,\, \Ddd.\]
\end{proof}

\begin{corollary}\label{cor3}
Assume $f$, $g$ satisfy  \mbox{\rm(H1)} and let $a\in\D$, $z,l\in\Dd$ be such  that
$V_1(a)_T<\infty$, $V_p(z)_T<\infty$ with
$x_0\geq l_0$. Set $x^n_0=x_0$, $k^n_0=0$ and
\[
\left\{\begin{array}{ll}
\Delta y^n_{(k+1)/n}&=f(x^n_{n/k})(a_{(k+1)/n}-a_{k/n})
+g(x^n_{k/n})(z_{(k+1)/n}-z_{k/n}), \\
x^n_{(k+1)/n}&=\max\big(x^n_{k/n}+\Delta y^n_{(k+1)/n},l_{(k+1)/n}\big),\\
k^n_{(k+1)/n}&=k^n_k+(x^n_{(k+1)/n}-x^n_{k/n})-\Delta
y^n_{(k+1)/n}
\end{array}
\right.
\]
and $x^n_t=x^n_{k/n}$, $k^n_t=k^n_{k/n}$, $l^n_t=l^n_{k/n}$,
$t\in[k/n,(k+1)/n)$, $k\in\N\cup\{0\}$. Then $\{(x^n,k^n)\}$  is
relatively compact in $\Ddd$  and its every limit point is a
solution of (\ref{eq3.1}). Consequently, equation (\ref{eq3.1})
has a solution (possibly nonunique).
\end{corollary}
\begin{proof}
It suffices to observe that $(x^n,k^n)$ is a solution of
(\ref{eq3.11}) with $a^n_t=a_{k/n}$, $z^n_t=z_{k/n}$,
$l^n_t=l_{k/n}$, $t\in[k/n,(k+1)/n)$, $k\in\N\cup\{0\}$. Also
observe that $\sup_n V_1(a^n)_T\leq V_1(a)_T<\infty$, $\sup_n
V_p(z^n)_T\leq V_p(z)_T<\infty$ for $T\in\Rp$  and
\[
(a^n,z^n,l^n)\longrightarrow (a,z,l)\quad\mbox{\rm in}\,\,\Dddt.
\]
Therefore the result follows from Theorem \ref{thm3}.
\end{proof}

\begin{theorem}\label{thm4}
Assume $f$, $g$ satisfy \mbox{\rm(H1)} and  \mbox{\rm(H2)}.
Then there exists a unique solution $(x,k)$ of (\ref{eq3.1}).
\end{theorem}
\begin{proof} Assume that there exist two solutions $(x^j,k^j)$,
$j=1,2$.\\
{\em Step 1.} We first replace (H2) by stronger condition
\begin{enumerate}
\item[(H2$^*$)] (a) $f:\Rd\rightarrow\Rd$  is Lipschitz continuous,
i.e. there is $L>0$ such that
\[
|f(x)-f(y)|\leq L|x-y|,\quad u\in\Rd.
\]
(b) $g:\Rd\rightarrow\Md$, its each component $g_{i,j}$ is
differentiable, \[\displaystyle{C=\max_{i,j}\sup_{x}|\nabla_xg_{i,j}(x)|}<\infty\] and there are
$\gamma\in(p-1,1]$ and $C_{\gamma}>0$ such that
\[
|\nabla_xg_{i,j}(x)-\nabla_xg_{i,j}(y)|\leq C_\gamma |x-y|^\gamma,\quad
x,y\in\Rd,\,\, i,j=1,\dots,d.
\]
\end{enumerate}
Fix $T\in\Rp$. By Corollary  \ref{cor1}, for any $t\leq T$ we have
\begin{align*}
\bar V_p(x^1-x^2)_t
&\leq (d+1)\bar V_p(\int_0^{\cdot}
f(x^1_{s-})-f(x^2_{s-})\,da_s)_t
\\
&\qquad+(d+1)\bar V_p(\int_0^{\cdot}g(x^1_{s-})-g(x^2_{s-})\,dz_s )_t.
\end{align*}
Moreover,
\[
\bar V_p(\int_0^{\cdot} f(x^1_{s-})-f(x^2_{s-})\,da_s)_t\leq
LV_1(a)_T\sup_{s\leq t}|x^1_s-x^2_s| \leq LV_1(a)_t\bar
V_p(x^1-x^2)_t
\]
and by (\ref{eq3.2}),
\[
\bar V_p(\int_0^{\cdot}g(x^1_{s-})-g(x^2_{s-})\,dy_s )_t
\leq C_{p,p/\gamma}\bar V_{p/\gamma}(g(x^1)-g(x^2))_tV_p(z)_t.
\]
By \cite[Theorem 2]{du}  for $i,j=1,\dots,d$ we have
\[V_{p/\gamma}(g_{i,j}(x^1)-g_{i,j}(x^2))_t
\leq C V_{p/\gamma}(x^1-x^2)_t
+C_\gamma\sup_{s\leq
t}|x^1_s-x^2_s| (V_p(x^1)_T)^\gamma.
\] Therefore
\begin{align*}
V_{p/\gamma}(g(x^1)-g(x^2))_t &\leq\s_{i,j=1}^{d}
V_{p/\gamma}(g_{i,j}(x^1)-g_{i,j}(x^2))_t\\
&\leq \tilde{C_1} V_{p/\gamma}(x^1-x^2)_t
+\tilde{C_2}\sup_{s\leq t}|x^1_s-x^2_s| (V_p(x^1)_T)^\gamma,
\end{align*}
where $\tilde{C_1}=C^{d^2}$, $\tilde C_2=(C_\gamma)^{d^2}$. Set
\[t_1=\inf\{t;\max[LV_1(a)_t,C_{p,p/\gamma}(\tilde{C_1}+\tilde{C_2}
V_p(x^1)_t^\gamma)V_p(z)_t]>\frac1{4(d+1)}\}\wedge T.\] Then $\bar
V_p(x^1-x^2)_{[0,t_1)}\leq\frac12\bar V_p(x^1-x^2)_{[0,t_1)}$,
thus $x^1=x^2$ on $[0,t_1)$. Since for $j=1,2$ we
have \[x^j_{t_1}=\max(x^j_{t_1-}
+f(x^j_{t_1-})\Delta a_{t_1}+g(x^j_{t_1-}) \Delta
y_{t_1},l_{t_1}),\] $x^1_{t_1}=x^2_{t_1}$, too. For $k\geq2$ set
\begin{align*}
t_k=\inf\{t>t_{k-1};&\max[LV_1(a)_{[t_{k-1},t)},\\
&C_{p,p/\gamma}(\tilde{C_1}+\tilde{C_2} V_p(x^1)_T^\gamma)V_p(z)_{[t_{k-1},t)}]
>\frac1{4(d+1)}\}\wedge T.
\end{align*}
Arguing as above we show recurrently that $x^1=x^2$ on each
interval $[t_{k-1},t_k]$. Since by the same arguments as in the
proof of Theorem \ref{thm3}, $m=\inf\{k;t_k=T\}$  is finite,
$x^1=x^2$ on the interval $[0,T]$,  which  completes the proof
under (H2$^*$). \\
{\em Step 2.} The general case. Set
$s_k=\inf\{t;\max(|x^1_t|,|x^2_t|\,)>k\}$, $k\in\N$. By  the first
part of the proof,
\[
x^1_t=x^2_t,\quad t<s_k,\,k\in\N.
\]
Since $s_k\to\infty$, this proves the corollary.
\end{proof}

\begin{corollary}\label{cor4}
Assume \mbox{\rm(H1)} and \mbox{\rm(H2)}. Let $a\in\D$ be such that $V_1(a)_T<\infty$,
$z$, $l\in\Dd$, with $V_p(z)_T<\infty$, $T\in\Rp$ and $x_0\geq l_0$. Let $\{(x^n,k^n)\}$ be a
sequence of approximations defined in Corollary \ref{cor3}.
Then
\begin{equation}
\label{eq3.16} (x^n,k^n)\lra (x,k)\quad\mbox{\rm
in}\,\,\Ddd\end{equation} and for any $T\in\Rp$,
\begin{equation}
\label{eq3.17}
\max_{k/n\leq T}|x^n_{k/n}-x_{k/n}|\lra0\quad
and\quad\max_{k/n\leq T}|k^n_{k/n}-k_{k/n}|\lra0,
\end{equation}
where $(x,k)$ is a unique solution of (\ref{eq3.1}).
\end{corollary}
\begin{proof} The convergence (\ref{eq3.16}) easily follows from
Corollary \ref{cor3}. If we set $x^{(n)}_t=x_{k/n}$,
$k^{(n)}_t=k_{k/n}$, $t\in[k/n,(k+1)/n)$, $k\in\N\cup\{0\}$, then
by \cite[Proposition 2.2]{js}
\[
(x^n,x^{(n)},k^n,k^{(n)})\lra (x,x,k,k)\quad\mbox{\rm in}\,\,\Ddddd
\]
Consequently,  $x^n-x^{(n)}\to0$   and $k^n-k^{(n)}\to0$  in
$\Dd$, which is  equivalent to (\ref{eq3.17}).
\end{proof}
\begin{corollary}
\label{cor5} Assume  \mbox{\rm(H1)}  and  \mbox{\rm(H2)}. Let
$a,z,l$ satisfy assumptions of Corollary \ref{cor4} and let
$\{a^n\}$ $\subset\D$,  $\{z^n\}$, $\{l^n\}$ $\subset\Dd$  be such
that for all $T\in\Rp$ $\sup_n V_1(a^n)_T<\infty$, $\sup_n V_p(z^n)_T<\infty$,
  and
\begin{equation}
\label{eq3.18} \sup_{t\leq T}(|a^n_t-a_t|+|z^n_t-z_t|+|l^n_t-l_t|)
\longrightarrow0,\quad T\in\Rp.
\end{equation}
If $\{(x^n,k^n)\}$ is a sequence of solutions of (\ref{eq3.11})  then
\[\sup_{t\leq T}(|x^n_t-x_t|+|k^n_t-k_t|)\longrightarrow0,\quad T\in\Rp,
\]
where $(x,k)$ is a unique  solution of (\ref{eq3.1}).
\end{corollary}
\begin{proof}
By Theorem \ref{thm3},
\[ (x^n,k^n,a^n,z^n,l^n)\longrightarrow(x,k,a,z,l) \in\Dddddj.\]
Since for any  $t\in\Rp$,
\[\Delta x_t\neq0\Longrightarrow
\Delta a_t\neq0\,\,\mbox{\rm or}\,\,\Delta z_t\neq0 \,\, \mbox{\rm
or}\,\,\Delta l_t\neq0,\] \cite[Lemma C]{s1}  shows that
$\sup_{t\leq T}|x^n_t-x_t|\to0$, $T\in\Rp$. Similarly we show the
uniform convergence of $k^n$ to $k$.
\end{proof}

\begin{corollary} \label{cor6}
Assume \mbox{\rm(H1)} and \mbox{\rm(H2)}. Let
$a,z,l$ satisfy assumptions of Corollary \ref{cor4}. Set $x^n_0=x_0$, $k^n_0=0$,
$t^n_0=0$, \[t^n_k=\inf\{t>t^n_{k-1}; \max(|\Delta a_t|,|\Delta
z_t|,|\Delta l_t|)>\frac1{n}\}\wedge (t^n_{k-1}+\frac1{n}),\quad
k\in\N,\] and
\[
\left\{\begin{array}{ll}
\Delta y^n_{t_{k+1}^n}&=f(x^n_{t^n_k})(a_{t_{k+1}^n}-a_{t_k^n})
+g(x^n_{t^n_k})(z_{t_{k+1}^n}-z_{t^n_k}), \\
x^n_{t_{k+1}^n}&=\max\big(x^n_{t^n_k}
+\Delta y^n_{t_{k+1}^n},l_{t_{k+1}^n}\big),\\
k^n_{t_{k+1}^n}&=k^n_{t_k^n}+(x^n_{t_{k+1}^n}-x^n_{t^n_k}) -\Delta
y^n_{t_{k+1}^n}
\end{array}
\right.
\]
and $x^n_t=x^n_{t_k^n}$, $k^n_t=k^n_{t_k^n}$,
$t\in[t_k^n,t_{k+1}^n)$, $k\in\N\cup\{0\}$. If $\{(x^n,k^n)\}$ is
a sequence of solutions of (\ref{eq3.11})  then
\[
\sup_{t\leq T}(|x^n_t-x_t|+|k^n_t-k_t|)\longrightarrow0,\quad T\in\Rp,
\]
where $(x,k)$ is a unique  solution of (\ref{eq3.1}).
\end{corollary}
\begin{proof}
Observe that $\{(x^n,k^n)\}$ is a sequence of solutions of
(\ref{eq3.11}) with  $a^n_t=a_{t_k^n}$, $z^n_t=z_{t_k^n}$,
$l^n_t=l_{t_k^n}$, $t\in[t_k^n,t_{k+1}^n)$, $k\in\N\cup\{0\}$, and
that
\[\sup_n V_1(a^n)_T\leq V_1(a)_T<\infty,\,\, \sup_n
V_p(z^n)_T\leq V_p(z)_T<\infty,\quad  T\in\Rp.\]
 Moreover, simple calculations show
that (\ref{eq3.18}) is satisfied. Therefore the desired result
follows from  Corollary \ref{cor5}.
\end{proof}

\section{SDEs with constraints}

Let $(\Omega, {\cal F}, ({\cal F}_t), P)$ be a filtered
probability space and let $A$ be an $({\cal F}_t)$ adapted process
with trajectories in $\D$, $Z,L$ be $({\cal F}_t)$ adapted
processes with trajectories in $\Dd$ such that  for any
$T\in\Rp$, $P(V_1(A)_T<\infty)=1$ and $P(V_p(Z)_T<\infty)=1$. Note that $Z$ need not to be a semimartingale.  However, it is a $p$-semimartigale and a Dirichlet process in the sense considered in \cite{k1,k2}  and \cite{cms}.

\begin{definition}\label{def1}
Let $X_0\geq L_0$. We  say that a
pair $(X,K)$  of  $({\cal F}_t)$ adapted processes with
trajectories in $\Dd$  such that $P(V_p(X)_T<\infty)=1$ for
$T\in\Rp$ is a strong solution of (\ref{eq1.1})  if
$(X,K)=SP_{L}(Y)$, where
\[
Y_t=X_0+\int_0^tf(X_{s-})\,dA_s+\int_0^tg(X_{s-})\,dZ_s,\quad t\in\Rp.
\]
\end{definition}

\begin{theorem} \label{thm5}
Assume \mbox{\rm(H1)}  and  \mbox{\rm(H2)}. If $X_0\geq L_0$ then
equation (\ref{eq1.1}) has a  unique strong solution $(X,K)$.
Moreover, if we define $(X^n,K^n)$ as
\[
X^n_t=X^n_{\tau_k^n},\quad K^n_t=K^n_{\tau_k^n},\quad
t\in[\tau_k^n,\tau_{k+1}^n),\quad k\in\N\cup\{0\},
\]
where $X^n_0=X_0$, $K^n_0=0$ and
\[
\left\{\begin{array}{ll}
\Delta Y^n_{\tau_{k+1}^n}&=f(X^n_{\tau^n_k})
(A_{\tau_{k+1}^n}-A_{\tau_k^n})+g(X^n_{\tau^n_k})
(Z_{\tau_{k+1}^n}-Z_{\tau^n_k}), \\
X^n_{\tau_{k+1}^n}&=\max\big(X^n_{\tau^n_k}
+\Delta Y^n_{\tau_{k+1}^n},L_{\tau_{k+1}^n}\big),\\
K^n_{\tau_{k+1}^n}&=K^n_{\tau_k^n}
+(X^n_{\tau_{k+1}^n}-X^n_{\tau^n_k})-\Delta Y^n_{\tau_{k+1}^n}
\end{array}
\right.
\]
with $\tau^n_0=0$, $\tau^n_k=\inf\{t>\tau^n_{k-1}; \max(|\Delta
A_t|,|\Delta Z_t|,|\Delta L_t|)>\frac1{n}\}\wedge
(\tau^n_{k-1}+\frac1{n})$, $n,k\in\N$, then for any $T\in\Rp$,
\[
\sup_{t\leq T}|X^n_t-X_t|\rightarrow0,\,\,\mbox{
P-a.s.},\quad\sup_{t\leq T}|K^n_t-K_t|\rightarrow0,\,\,
\mbox{P-a.s.},
\]
\end{theorem}
\begin{proof}
From Theorem \ref{thm4} we deduce that for every $\omega\in\Omega$
there exists a unique solution
$(X(\omega),K(\omega))=SP_{L(\omega)}(Y(\omega))$. Moreover, by
Corollary \ref{cor6}, for every $\omega\in\Omega$ and $T\in\Rp$,
\[
\sup_{t\leq T}|X^n_t(\omega)-X_t(\omega)|\rightarrow0, \quad
\sup_{t\leq T}|K^n_t(\omega)-K_t(\omega)|\rightarrow0.
\]
Since for any $n\in\N$ the pair $(X^n,K^n)$  is $({\cal F}_t)$
adapted, the pair of limit processes  $(X,K)$  is $({\cal F}_t)$
adapted as well, which completes the proof.
\end{proof}

\begin{corollary} \label{cor7}
Under the assumptions of Theorem \ref{thm5} with random sequences
of partitions $\{\tau^k_n\}$ replaced by  constant sequences
$\{\frac{k}{n}\}$, $k\in\N\cup\{0\}$, $n\in\N$, we have
\[
(X^n,K^n)\longrightarrow (X,K)\quad \mbox{P-a.s.}\,\,in\,\, \Ddd,
\]
where $(X,K)$ is a unique strong solution of (\ref{eq1.1}).
\end{corollary}
\begin{proof}
It suffices to apply Corollary \ref{cor4}.
\end{proof}

Let $\BH$ be a fractional Brownian  motion (fBm) with Hurst index
$H>1/2$, i.e. a continuous centered Gaussian process with
covariance
\[EB^H_{t_2}B^H_{t_1}=\frac12(t_2^{2H}+t_1^{2H}-|t_2-t_1|^{2H}),
\quad t_1,t_2\in\RRp.
\]

Let  $\BB=\int_0^\cdot\sigma_s\,d\BH_s$, where $\sigma:\RRp\to\RR$ is
 a measurable function  such that
$
\|\sigma\|_{\LH{[0,T]}}:=(\int_0^T|\sigma_s|^{1/H}ds)^H$
 $<\infty$, $T\in\Rp$. Then $\BB$ is
also  a continuous centered Gaussian process with continuous
trajectories. Moreover, if
$p>1/H$ then
\begin{equation}
\label{eq4.10} P(V_p(\BB)_T<\infty)=1,\quad T\in\Rp
\end{equation}
(see, e.g., \cite[Proposition 2.1]{fsz}).
Note also that $Z^H$  is a Dirichlet process from the class ${\cal
D}^{1/H}$ studied in \cite{cs}.

We now show how to apply our results to fractional SDEs with
constraints of the form (\ref{eq1.4}). Let
$B^H=(B^{H,1},\dots,B^{H,d})$, where  $B^{H,1},\dots,B^{H,d}$ are
independent  fractional Brownian motions, and let
$Z^H=(Z^{H,1},\dots,Z^{H,d})$, where
$Z^{H,i}=\int_0^\cdot\sigma^i_s\,d B^{H,i}_s$ with
$\sigma^i:\RRp\to\RR$ such that
$\|\sigma^i\|_{\LH{[0,T]}}<\infty$, $T>0$, $i=1,\dots,d$.

\begin{corollary} \label{cor10}
Assume \mbox{\rm(H1)} and  \mbox{\rm(H2)}. If $X_0\geq L_0$ then
equation (\ref{eq1.4}) has a    unique strong solution $(X,K)$.
Moreover, if
\[
X^n_t=X^n_{k/n},\quad K^n_t=K^n_{k/n},\quad
t\in[k/n,(k+1)/n),\quad k\in\N\cup\{0\},\,n\in\N,
\]
where $X^n_0=X_0$, $K^n_0=0$ and
\[\left\{\begin{array}{ll}
\Delta Y^n_{(k+1)/n}&=f(X^n_{k/n})(a_{(k+1)/n}-a_{k/n})
+g(X^n_{k/n})(\ZH_{(k+1)/n}-\ZH_{k/n}), \\
X^n_{(k+1)/n}&=\max\big(X^n_{k/n}+\Delta Y^n_{(k+1)/n},L_{(k+1)/n}\big),\\
K^n_{(k+1)/n}&=K^n_{k/n}+(X^n_{(k+1)/n}-X^n_{k/n})-\Delta
Y^n_{(k+1)/n},
\end{array}
\right.
\]
then for any $T\in\Rp$,
\[\sup_{t\leq T}|X^n_t-X_t|\rightarrow0,\,\,
P\mbox{-a.s.},\quad \sup_{t\leq T}|K^n_t-K_t|\rightarrow0,\,\,
P\mbox{-a.s.},
\]
\end{corollary}
\begin{proof}
It suffices to apply Corollary \ref{cor7}  and use the facts that
$a$  is a continuous function  and $\ZH$ has continuous
trajectories.
\end{proof}

\begin{remark}
To  approximate  $Z^H$ one can use the methods developed in
\cite{fsz,sz2}.
\end{remark}

\section{Appendix}

\begin{proposition}
\label{prop1}
Let  $\{x^n\}\subset\DMd$, $\{z^n\}\subset\Dd$  and
\[\sup_n \bar V_q(x^n)_T<\infty,\quad \sup_n V_p(z^n)_T<\infty,\quad T\in\Rp,\]
where $1/p+1/q>1$, $p,q\geq1$. If
$
\{z^n\}$ is relatively compact in $\Dd$
then
\[
\{(\,\int_0^{\cdot}x^n_{s-}\,dz_s^n,\,z^n\,)\}
\quad is \,\,relatively\,\, compact\,\, in\,\,\Ddd.
\]
\end{proposition}
\begin{proof}
Without loss of generality we may and will assume that
\[
z^n\lra z\quad\mbox{\rm in}\quad\Dd.
\]
We follow arguments from the proof of 
\cite[Proposition 3.3]{ms} and \cite[Proposition 2]{s1}. Let $t^n_{k,0}= 0$, $ t^n_{k,i+1}=\min
( t^n_{k,i} + \delta_{k,i}, \inf \{ t> t^n_{k,i},|\Delta z^n_t|>
\delta_k \})$ and $t_{k,0}= 0$, $ t_{k,i+1}=\min ( t_{k,i} +
\delta_{k,i}, \inf \{t> t_{k,i},|\Delta z_t|> \delta_k \})$,
where $\{\delta_k\}$, $\{\{ \delta_{k,i} \}\}$ are families of
constants such that  $\delta _k \downarrow 0$, $|\Delta z_t| \neq
\delta_k , t \in \Rp $, $\delta_k /2 \leq \delta_{k,i} \leq
\delta_k$ and $|\Delta z_{t_{k,i} + \delta_{k,i}}|=0$,
$i\in\N\cup\{0\}$, $k,n\in\N$. Define
\[
z^{n,(k)}_t=z^n_{t^n_{k,i}},\,\, t\in [t^n_{k,i},t^n_{k,i+1})
\quad\mbox{\rm and}\quad z^{(k)}_t=z_{t_{k,i}}, \,\,t\in [t_{k,i},t_{k,i+1})
\]
for $i\in\N\cup\{0\}$, $n,k\in\N$.  Then $ V_p(z^{n,(k)})_T \leq
V_p(z^n)_T$, $n\in\N$, $ V_p(z^{(k)})_T \leq V_p(z)_T$ for
$T\in\Rp$ and
\begin{equation}\label{eq3.3}
t^n_{k,i}\to t_{k,i},\quad z^n_{t^n_{k,i}}\to z_{t_{k,i}}, \quad
i\in\N\cup\{0\}, k\in\N.
\end{equation} Consequently,
\begin{equation}
\label{eq3.4}\sup_nV_p(z^n-z^{n,(k)})_T<\infty,\quad k\in\N.
\end{equation}
and
\begin{equation}\label{eq3.5}
(\,z^{n,(k)},\,z^n\,)
\longrightarrow\,(\,z^{(k)},\,z\,)
\quad\mbox{\rm  in}\,\,\Ddd,\,\, k\in\N.
\end{equation}
Moreover,
\begin{equation}\label{eq3.6}\sup_{t\leq T}
|z^{(k)}_t-z_t|\lra0,\quad T\in\Rp,
\end{equation}
which together with (\ref{eq3.5}) implies that
\begin{equation}\label{eq3.7}
\lim_{k\to\infty}\limsup_{n\to\infty}\sup_{t\leq T}
|z^{n,(k)}_t-z^{n}_t|=0,\quad T\in\Rp.
\end{equation}
On the other hand, $\int_0^t x^{n}_{s-} dz^{n,(k)}_s=\sum _{j\leq
i}x^n_{t^n_{k,j}-}(z^n_{t^n_{k,j}}-z^n_{t^n_{k,j-1}})$,
$t\in[t^n_{k,i},t^n_{k,i+1})$.  Using  (\ref{eq3.3}),
(\ref{eq3.5}) and the fact that   $\sup_n\bar V_q(x^n)_T<\infty$
implies  that $\{\sup_{t\leq T}|x^n_t|\}$  is bounded,  we
conclude that for any $k\in\N$,
\begin{equation}\label{eq3.8}
\{(\int_0^{\cdot}x^n_{s-}dz^{n,(k)}_s,z^n)\}\quad\mbox{\rm is
relatively compact in}\,\, \Ddd.
\end{equation}
Let $p'>p$ be such that $1/p'+1/q>1$. By
(\ref{eq3.2})
\[
\sup_{t\leq T}|\int_0^tx^{n}_{s-}d(z^n-z^{n,(k)})_s| \leq
C_{p',q}\bar V_{q}(x^{n})_TV_{p'}(z^n-z^{n,(k)})_T.
\]
Moreover
\[
V_{p'}(z^{n}-z^{n,(k)})_T \leq \mbox{\rm
Osc}(z^n-z^{n,(k)})^{1-p/p'}_TV_p(z^n-z^{n,(k)})_T^{p/p'},
\]
where ${\rm Osc}(x)_T=\sup_{s,t\leq T}|x_t-x_s|$. Since \[\mbox{\rm
Osc}(z^n-z^{n,(k)})_T\leq 2\sup_{t\leq T}|z^n-z^{n,(k)}_t|\lra0,\]
we deduce from the above  that
\begin{equation}
\label{eq3.9} \lim_{k\to\infty}\limsup_{n\to\infty} \sup_{t\leq T}
|\int_0^tx^n_{s-}\,d(z^n-z^{n,(k)})_s|=0,\quad T\in\Rp.
\end{equation}
Combining (\ref{eq3.8})  with   (\ref{eq3.9})  we get the desired
result.
\end{proof}
\begin{corollary}\label{cora1} Let $\{a^n\}\subset \D$,
$\{z^n\}$,$\{y^{n}\}$   $\subset\Dd$ and
$\{x^{n}\}$ $\subset\DMd$ be sequences of functions such that
\[\sup_n\max( V_1(a^n)_T, V_p(z^n)_T, \bar V_q(y^{n})_T, \bar V_q(x^{n})_T)<\infty, \quad T\in\Rp,\] where
$1/p+1/q>1$, $p,q\geq1$. If
\[ \{(a^n,\,z^n)\}\quad is\,\, relatively\,\,
compact\,\, in\,\,\Dpd\] then
\[
\{(\,\int_0^{\cdot}y^{n}_{s-}\,da^n_s,\,a^n,\,\int_0^{\cdot}x^{n}_{s-}\,dz_s^n,\,z^n\,)\}
\quad is \,\,relatively\,\, compact\,\, in\,\,\Dpddd.
\]
\end{corollary}
\begin{proof} Set $\bar d=2d$  and for every $n\in\N$ define  $\bar  z^n\in\Dbd$ and
 $\bar x^n\in\Dbdwad$
 by the formulas
  \[\bar z^{n,i}=\left\{\begin{array}{ll}
a^n,&i=1,...,d, \\
z^{n,i-d},&i=d+1,...,2d
\end{array}
\right.
\]
and
 \[(\bar x^{n})_{i,j}=\left\{\begin{array}{ll}
y^{n,i},&i=j=1,...,d, \\
(x^{n})_{i-d,j-d},&i=d+1,...,2d,\,j=d+1,...,2d,\\
0,& \mbox{\rm otherwise.}
\end{array}
\right.
\]
By Proposition \ref{prop1}, 
\[ \{(\,\int_0^{\cdot}\bar
x^n_{s-}\,d\bar z_s^n,\,\bar z^n\,)\} \quad is \,\,relatively\,\,
compact\,\, in\,\,\Dbdd,
\]
from which one can deduce the corollary.
\end{proof}
\begin{proposition}
\label{prop2} Let  $\{x^n\}\subset\DMd$, $\{z^n\}\subset\Dd$
 and
 \[\sup_n \bar V_q(x^n)_T<\infty,\quad\sup_n
V_p(z^n)_T<\infty, \quad T\in\Rp,\] where $1/p+1/q>1$, $p,q\geq1$.  If
$
(\,x^n,z^n\,)\longrightarrow\,(\,x,z\,)$ in  $\Ddwad$
then
\[
(\,x^n,z^n,\,\int_0^{\cdot}x_{s-}^n\,dz_s^n\,)
\longrightarrow\,(\,x,z,\,\int_0^{\cdot}x_{s-}\,dz_s\,) \quad  in
\,\,\Ddwadd.
\]
\end{proposition}
\begin{proof}
Since $V_q(x)_T\leq\liminf_{n\to\infty} V_q(x^n)_T<\infty$ and
 $V_p(z)_T\leq\liminf_{n\to\infty} V_p(z^n)_T <\infty$
for $T\in\Rp$, the integral $\int_0^\cdot x_{s-} dz_s$ is well
defined. Let $\{\{z^{n,(k)}\}\}$, $\{z^{(k)}\} $ be  families of
functions defined in the proof of Proposition \ref{prop1}. It
follows from the equality $\int_0^t x_{s-} dz^{(k)}_s=\sum_{j\leq
i}x_{t_{k,j}-}(z_{t_{k,j}}-z_{t_{k,j-1}})$,
$t\in[t_{k,i},t_{k,i+1})$, and (\ref{eq3.3}), (\ref{eq3.5}) that
for any $k\in\N$,
\begin{equation}\label{eq3.10}
(x^n,z^n,\int_0^{\cdot}x_{s-}^{n}\,dz_s^{n,(k)})
\longrightarrow\,(x,z,\int_0^{\cdot}x_{s-}\,dz^{(k)}_s)\quad\mbox{\rm
in}\,\,\Ddwadd.
 \end{equation}
As in the proof of (\ref{eq3.9}) we check that
$\lim_{k\to\infty}\sup_{t\leq T}|\int_0^tx_{s-}d(z-z^{(k)})_s|=0$,
$T\in\Rp$. From this and  (\ref{eq3.9}), (\ref{eq3.10}) the result
follows.
\end{proof}
Using arguments from the proof of Corollary \ref{cora1} it is easy
to check that Proposition \ref{prop2} implies the following
corollary.
\begin{corollary}\label{cora2} Let $\{a^n\}\subset \D$,
$\{z^n\}$,$\{y^{n}\}$   $\subset\Dd$ and
$\{x^{n}\}$  $\subset\DMd$ be sequences of functions such that
\[\sup_n \max(V_1(a^n)_T, V_p(z^n)_T, \bar V_q(y^{n})_T, \bar V_q(x^{n})_T)<\infty,\quad  T\in\Rp,\]
 where
$1/p+1/q>1$, $p,q\geq1$. If \[
(y^{n},\,a^n,\,x^{n},\,z^n)\longrightarrow(y,\,a,\,x,\,z)\quad
in\,\,\Dppd\]
then
\[
(\int_0^{\cdot}y^{n}_{s-}\,da^n_s,\, a^n,\, \int_0^{\cdot}x^{n}_{s-}\,dz_s^n,\, z^n)\longrightarrow
(\int_0^{\cdot}y_{s-}\,da_s,\, a, \,\int_0^{\cdot}x_{s-}\,dz_s, \,z) \]
in $\Dpppd$.
\end{corollary}

\end{document}